\documentclass{birkmult}
\usepackage{showkeys}
 \newtheorem{thm}{Theorem}[section]
 \newtheorem{cor}[thm]{Corollary}
 \newtheorem{lem}[thm]{Lemma}
 \newtheorem{prop}[thm]{Proposition}
 \theoremstyle{definition}
 
 \theoremstyle{remark}
 \newtheorem{rem}[thm]{Remark}
 
 \numberwithin{equation}{section}
\newcommand{\M}{\mathcal{M}}

\newcommand{\extr}{{\rm extr\kern 2pt}}
\newcommand{\supp}{{\rm supp\kern 2pt}}

\begin{document}
%-------------------------------------------------------------------------
% editorial commands: to be inserted by the editorial office
%
%\firstpage{1}
%\volume{228}
%\DOI{003-0001}
%\seriesextraline{This is the Concrete Title of this Book\br H.E. R and S.T.C. W, Eds.}
%\firstpage{1}
\title[ Logarithmic Submajorization  ]
 {Logarithmic submajorization, uniform majorization and H\"older type inequalities for $\tau$-measurable operators}
%----------Author 1
\author{PG Dodds}
\address{%
\\
College of Science and Engineering,
Flinders University\\
GPO Box 2100,
Adelaide 5001, Australia}

\email{peter@csem.flinders.edu.au}

%\thanks{This work was partially supported by the Australian Research Council.}
%----------Author 2
\author{TK Dodds}
\address{%
\\
College of Science and Engineering,
Flinders University\\
GPO Box 2100,
Adelaide 5001, Australia}

\email{theresa@csem.flinders.edu.au}

%----------Author 3
\author{FA Sukochev}
\address{%
\\
School of Mathematics and Statistics,
UNSW\\
Kensington, NSW 2052, Australia\\
}
\email{f.sukochev@nsw.edu.au}
\thanks{ The work of the third and fourth  authors was  supported by the Australian Research Council.}

%----------Author 4
\author{ D.Zanin}
\address{%
\\
School of Mathematics and Statistics,
UNSW\\
Kensington, NSW 2052, Australia\\}

\email{d.zanin@nsw.edu.au}

%\thanks{The work of the fourth author was supported by the Australian Research Council.}


%\date{August 2015}
\dedicatory{Dedicated to the memory of W.A.J. Luxemburg}
%%% ----------------------------------------------------------------------

\begin{abstract}We extend the notion of the determinant function $\Lambda$, originally introduced by T.Fack for $\tau$-compact operators, to a natural algebra of $\tau$-measurable operators affiliated with a semifinite von Neumann algebra which coincides with that defined by Haagerup and Schultz in the finite case and on which the determinant function is shown to be submultiplicative. Application is given to H\"older type inequalities via general Araki-Lieb-Thirring inequalities due to Kosaki and Han and to a Weyl-type theorem for uniform majorization. 
\end{abstract}

%%% ----------------------------------------------------------------------
\maketitle
%%% ----------------------------------------------------------------------
%\tableofcontents
%\section{Document Preamble}

%\section{Introduction} 

%\centerline{P.G.Dodds, T. K. Dodds, F.A.Sukochev and D.Zanin}
%\centerline{August 2015}
\section{Introduction }
\bigskip
In their study of Brown measures of unbounded operators affiliated with a finite von Neumann algebra $({\mathcal M},\tau)$, Haagerup and Schultz extended the classical 
Fuglede-Kadison determinant
\begin{equation*}
\Delta(x)
=\exp\left(\int_0^{\tau({\bf 1})}\log t\ d(\tau e^{\vert x\vert})(s)\right),\quad x\in {\mathcal M}
\end{equation*} 
to a multiplicative functional $\Delta$ on the subalgebra ${\mathcal M}^\Delta$ consisting of those $\tau$-measurable operators $x$ affiliated with ${\mathcal M}$ for which
\begin{equation*}
\tau(\log_+\vert x\vert)=\int_0^\infty\log_+td(\tau e^{\vert x\vert}) (t)<\infty
\end{equation*}
Here $e^{\vert x\vert} (\cdot)=\chi_{(\cdot)}(\vert x\vert)$ is the spectral measure of $\vert x\vert $ and $\tau e^{\vert x\vert }(\cdot)=\tau(e^{\vert x\vert}(\cdot))$ is the corresponding Borel measure supported on the spectrum $\sigma(\vert x\vert)$ of $\vert x\vert$.

A principal aim of this paper is to extend the Haagerup-Schultz approach from the setting of finite von Neumann algebras to the semifinite case. Our approach is via the determinant function $\Lambda(\cdot)$ given by setting 
\begin{equation*}
\Lambda(t;x)=\exp \left(\int_0^t \log \mu(s;x)\, ds \right),\quad t>0,\quad x\in L_{\log_+}(\tau).
\end{equation*}
Here $\mu(\cdot;x)$ is the generalised singular value function (see below)
and $L_{\log_+}(\tau)$ is the natural extension of the algebra ${\mathcal M}^\Delta$ of Haagerup-Schults to the general semifinite setting. The determinant function $\Lambda$ goes back to A.Grothendieck~\cite{G} and was studied by Fack ~\cite{Fa},\cite{Fa83} in the case  of $\tau$-compact operators and subsequently by Fack and Kosaki~\cite{FaKo1986}
for operators satisfying a "Lorentz space"-type condition(see the definition of this condition in the domments following Proposition 3.1 below). 

Basic properties of  the class $L_{\log_+}(\tau)$ are studied in the third section where it is noted that this class is an algebra properly containing the space of $\tau$-compact operators as well as those $\tau$-measurable operators
satisfying a "Lorentz space"-type condition. One of the principal results of the paper is given in Theorem~\ref{thmFK} which establishes the submultiplicativity of the determinant $\Delta$ on the algebra $L_{\log_+}(\tau)$. A principal ingredient of the proof is the multiplicativity of the Haagerup-Schultz
functional $\Delta$ on the algebra ${\mathcal M}^\Delta$. 

In subsequent sections, attention is directed to H\"older type
inequalities for $\tau$-measurable operators. Principal results in Section 5 include a submajorization
inequality of H\"older type given in Proposition ~\ref{propKIp3} and 
corresponding symmetric norm inequalities given in
Proposition~\ref{propKIp4}, which are established using techniques based on logarithmic
submajorization. See, for example, Hiai ~\cite{Hi1997} in the 
finite-dimensional setting. While  these results extend inequalities for
unitarily invariant norms given in ~\cite{Ko1998} 
Theorem 3, the techniques given there do
not extend to the more general setting and a 
crucial tool  is a strengthened version ~\cite{Ha} given in
Proposition~\ref{propalt1} of a
submajorization inequality of Araki-Lieb-Thirring type due to 
Kosaki~\cite{Ko1992} in the setting of semi-finite von Neumann algebras and
based on properties of the Fuglede-Kadison determinant. This section then
concludes with a convexity inequality (Proposition~\ref{prophiai3}) which is 
due to Hiai and Zhan [HZ] in the
matrix setting for unitarily invariant norms.

We 
show in Section 6 (Theorem~\ref{thmWI}) that Weyl's submajorization 
inequality for the singular
values of a product (~\cite{FaKo1986} Theorem 4.2(iii)) continues to hold for
uniform majorization, a notion introduced by Kalton
and Sukochev ~\cite{KS}
 which is a strengthening of the classical notion of
submajorization in the sense of Hardy, Littlewood and Polya.  This permits us  to give a simple proof of generalized
H\"older inequalities for general symmetric spaces proved recently by
Bekjan and Ospanov~\cite{BO} and Albadawi ~\cite{Al}.

{\it This paper is dedicated to the memory of Wim Luxemburg. It is with great affection that the first two authors recall his Caltech lectures on the theory of integration which introduced them to the paper of Grothendieck, and to the ideas of decreasing rearrangements of operators, and the determinant function 
which plays a central role in this paper.}

\section{Preliminaries and notation}
Throughout this paper $\M\subseteq {\mathcal B}({\mathcal H})$
will denote a semifinite von Neumann algebra on some Hilbert space ${\mathcal
H}$ (here, ${\mathcal B}({\mathcal H})$ is the algebra of all
bounded linear operators on ${\mathcal H}$ equipped with the
operator norm denoted by $\|\cdot\|_\infty$ or
$\|\cdot\|_{{\mathcal B}({\mathcal H})}$). Unless otherwise
stated, it will be assumed throughout that $\M$ is equipped with a
fixed faithful normal semifinite trace $\tau$. For standard facts
concerning von Neumann algebras,
 we refer to \cite{Di}, \cite{Ta}.
 The identity in $\M$ \  is denoted
by ${\textbf 1}$ and we denote by ${\mathcal P}\left({\mathcal
M}\right)$ the complete lattice of all (self-adjoint) projections
in $\M$. A linear operator $x:{\mathcal D}(x)\to {\mathcal H} $,
with domain ${\mathcal D}(x)\subseteq {\mathcal H}$, is said to be
{\it affiliated with}  $\M$ \ if $yx\subseteq xy$ for all $y$ in the commutant
$\M^\prime$ of $\M$ (equivalently, $ux=xu$  for all unitary $u$ in
 $\M^\prime$). . For any self-adjoint operator
$x$ on $\mathcal{H}$, its spectral measure is
denoted by $e^{x}$. A self-adjoint operator $x$ is affiliated with $\mathcal{%
M}$ if and only if $e^{x}\left( B\right) \in \mathcal{P}\left( \mathcal{M}%
\right) $ for any Borel set $B\subseteq \mathbb{R}$.

The closed and densely defined operator $x$, affiliated with
$\M$, is called $\tau$-{\it measurable} if and only if there
exists a number $s\geq 0$ such that
$$
\tau\left(e^{\left\vert x\right\vert}(s,\infty) \right)<\infty.
$$
 The collection of all $\tau$-measurable operators is denoted by $S(\tau)$.
 With the sum and product defined as the respective closures of the algebraic sum and
product, it is well known that $S(\tau)$ \ is a *-algebra. For
$\epsilon,\ \delta>0$, we denote by $V(\epsilon,\delta)$ the set
of all $x\in S(\tau)$ for which there exists an orthogonal
projection $p\in P(\M)$ such that $p({\mathcal
H})\subseteq{\mathcal D}(x),\ \Vert xp\Vert _{{\mathcal
B}({\mathcal H})}\leq \epsilon$ and  $\tau ({\textbf 1}-p)\leq
\delta$. The sets $\{V(\epsilon,\delta):\epsilon,\ \delta >0\}$
form a base at $0$ for a metrizable Hausdorff topology on
$S(\tau)$, which is called the {\it measure topology}. Equipped
with this topology, $S(\tau)$ is a complete topological
$*$-algebra. These facts and their proofs can be found in the
papers \cite{Ne},~\cite{Te}. See also ~\cite{DPS}

The collection of all closed, densely defined operators $x$ in
$H$, affiliated with the von Neumann algebra $\mathcal{M}$ and
satisfying $\tau \left( e^{\left\vert x\right\vert }\left(
\lambda,\infty \right) \right) <\infty $ for all $\lambda >0$,
will be denoted by $S_{0}\left( \tau \right) $. The elements of
$S_{0}\left( \tau \right)$ are sometimes called $\tau
$\textit{-compact operators}. Evidently, each $x\in S_{0}\left(
\tau \right) $, is $\tau $-measurable, that is, $S_{0}\left( \tau
\right) \subseteq S\left( \tau \right) $.

For $x\in S(\tau)$, the {\it (generalised)
singular value function} $\mu(x)=\mu(\cdot;x)=\mu(\cdot ;\vert x\vert)$ is
defined by
$$
\mu \left( t;x\right) :=\inf \left\{ s\geq 0:d(s;|x|) \leq
t\right\} , \quad t\geq 0,
$$
where
$$
d(s;|x|):=\tau \left( e^{\left\vert x\right\vert }\left( s,\infty
\right) \right), \quad s\geq 0.
$$
It follows directly that the singular value function $\mu(x)$ is a
decreasing, right-continuous function on the positive half-line
$[0,\infty)$. Moreover, $\mu(uxv)\leq \Vert u\Vert_\infty \Vert v\Vert_\infty
\mu(x)$ for all $u,v\in \M$ and $x\in S(\tau)$ and
$\mu(f(x))=f(\mu(x))$
 whenever $0\leq x\in S(\tau)$ and $f$ is  an
increasing continuous function on $[0,\infty)$ which satisfies $f(0)=0$.

It should be observed that a sequence $\left\{ x_{n}\right\} _{n=1}^{\infty
}$ in $S\left( \tau \right) $ converges to zero for the measure topology if
and only if $\mu \left( t;x_{n}\right) \rightarrow 0$ as $n\rightarrow
\infty $ for all $t>0$.

If  $m$ denotes Lebesgue measure on the semiaxis $[0,\infty)$, and
if we consider $L^\infty(m)$ as an Abelian von Neumann algebra
acting via multiplication on the Hilbert space   ${\mathcal
H}=L^2(m)$, with the trace given by integration with respect to
$m$, then $S(m) $ consists of all measurable functions on
$[0,\infty)$ which are bounded except on a set of finite measure.
In this case for $f\in S(m)$, the generalized singular value
function $\mu(f)$ is precisely the classical decreasing
rearrangement of the function $\vert f\vert$,  which is usually
denoted by $f^*$. In this setting, convergence for the measure
topology coincides with the usual notion of convergence in
measure.

If $\M ={\mathcal B}({\mathcal H})$ and $\tau $ is the standard
trace, then $S(\tau) =\M$, the measure topology coincides with the
operator norm topology. If $x\in S(\tau) $, then $x$ is compact if and only if
 $\lim_{t\to \infty}\mu(t;x)=0$; in this case,
$$
 \mu _n(x)=\mu \left(t;x\right), \quad t\in [n,n+1),\quad  n=0,1,2,\dots ,
$$
and the sequence $\{\mu _n(x)\}_{_{n=0}}^{\infty }$ is just
the sequence of eigenvalues of $\vert x\vert $ in non-increasing order
 and counted according to multiplicity.

The real vector space $S_{h}\left( \tau \right) =\left\{ x\in S\left( \tau
\right) :x=x^{\ast }\right\} $ is a partially ordered vector space with the
ordering defined by setting $x\geq 0$ if and only if $\left\langle x\xi ,\xi
\right\rangle \geq 0$ for all $\xi \in \mathcal{D}\left( x\right) $. The
positive cone in $S_{h}\left( \tau \right) $ will be denoted by $S\left(
\tau \right) ^{+}$. If $0\leq x_{\alpha }\uparrow _{\alpha }\leq x$ holds in
$S\left( \tau \right) ^{+}$, then $\sup_{\alpha }x_{\alpha }$ exists in $%
S\left( \tau \right) ^{+}$. The trace $\tau $ extends to $S\left( \tau
\right) ^{+}$ as a non-negative extended real-valued functional which is
positively homogeneous, additive, unitarily invariant and normal. This
extension is given by
\[
\tau \left( x\right) =\int_{0}^{\infty }\mu \left( t;x\right) dt,\ \ \ x\in
S\left( \tau \right) ^{+},
\]
and satisfies $\tau \left( x^{\ast }x\right) =\tau \left( xx^{\ast }\right) $
for all $x\in S\left( \tau \right) $. It should be observed that if $f$ is
an increasing continuous function on $\left[ 0,\infty \right) $ satisfying 
$%
f\left( 0\right) \geq 0$, then
\begin{equation}
\label{eqn001}
f(\mu(x))=\mu(f(\vert x\vert))
\end{equation}
and if $f(0)=0$ then
\begin{equation}
\tau \left( f\left( \left\vert x\right\vert \right) \right)
=\int_{0}^{\infty }\mu \left( t;f\left( \left\vert x\right\vert \right)
\right) dt=\int_{0}^{\infty }f\left( \mu \left( t;x\right) \right) dt
\label{eq01}
\end{equation}
for all $x\in S\left( \tau \right) $.

%If $({\mathcal N},\sigma) $ is a semifinite von Neumann algebra,
If $x,y\in S(\tau)$ 
%and $y\in S(\sigma)$ 
then $x$ is said to be {\it submajorised}
 by $y$ (in the
sense of Hardy, Littlewood and Polya) if and only if
$$
\int_0^t\mu(s;x)ds\leq
\int_0^t\mu(s;y)ds
$$
for all $t\geq 0$. We write $x\prec\prec y$, or equivalently, $\mu(x)\prec\prec
\mu(y)$.

A linear subspace
$E\subseteq S(\tau)$, equipped with a norm $\Vert \cdot \Vert _E$ will be
called \textit{symmetrically normed} if $x\in E$, $y\in S\left( \tau
\right) $ and $\mu \left( y\right) \leq \mu \left( x\right) $ imply that $%
y\in E$ and $\left\Vert y\right\Vert _{E}\leq \left\Vert x\right\Vert _{E}$.
If a symmetrically normed space is Banach, then it will
be simply called a 
\textit{
 symmetric space}. If $\M$ is $L_\infty(m)$,
then a symmetrically normed space
$E\subseteq S(m)$ will be called, for simplicity, a symmetrically
normed space on $[0,\infty)$. If $E\subseteq S(\tau)$ is a symmetrically normed space, then the
embedding of $E$ into $S(\tau)$ is continuous from the norm topology of $E$ to
the measure topology on $S(\tau)$.

A wide class of non-commutative symmetrically normed
spaces may be
 constructed
 as follows. If $E\subseteq S(m)$ is  
  symmetrically normed space on
 $[0,\infty)$, set
 \[
E\left( \tau \right) =\left\{ x\in S\left( \tau \right) :\mu \left( x\right)
\in E\right\},\quad \Vert x\Vert _{E(\tau)}:=\Vert \mu(x)\Vert _E
\]
It may be shown (\cite{KS}; see also \cite{DDP1989a}, ~\cite{DP2}), that
$(E(\tau),\Vert \cdot\Vert _{E(\tau)})$ is  
symmetrically normed and is a Banach space if $E$ is a
 Banach space. 
 
 Important special cases of 
 this
 construction occur when $E$ is taken to be the familiar Lebesgue space 
 $L_p, 1\leq p\leq \infty$. In this
 paricular case, the spaces $E(\tau)=L_p(\tau)$ are the familiar non-commutative
 $L_p$-spaces and in this setting we have that $L_\infty(\tau)={\mathcal M}$.
  In the special case that ${\mathcal M}$ is $
 B({\mathcal H})$ the corresponding non-commutative $L_p$ spaces are the
 familiar Schatten classes ${\mathfrak S}_p$. As is well-known, the space
 $L_1(\tau)$ may be identified with the von Neumann algebra predual ${\mathcal
 M}_*$ of
 ${\mathcal M}$ with respect to trace duality. We observe that, whenever
 $E\subseteq S(m)$ is a symmetric space, then the continuous
 inclusions
 \begin{equation*}
 L_1(\tau)\cap{\mathcal M}
 \subseteq E(\tau)\subseteq L_1(\tau)+{\mathcal M}
 \end{equation*}
 hold and that, if $x\in S(\tau)$, then $x\in L_1(\tau)+{\mathcal M}$ if and
 only if $\int_0^t\mu(s;x)ds<\infty$ for at least one $t>0$ or, equivalently, for all $t>0$.
 
 For further details and proofs, we refer the
 reader to  ~\cite{DDP1989a}, ~\cite{DDP1993},  ~\cite{KS}, 
  ~\cite{DP2},\cite{DPS},\cite{LSZ}.

\section{The algebra $L_{\log_+}(\tau)$}

\bigskip
Throughout, we will write $L_1,\ L_\infty$  rather than $L_1(m),L_\infty(m)$ for the usual Lebesgue spaces on the semiaxis $[0,\infty)$ equipped with Lebesgue measure $m$.\hfil\break
We set
\begin{equation}
\label{eqnFK1}
L_{\log_+}(\tau)
:=\{x\in S(\tau):\log_+\vert x\vert \in L_1(\tau)+\mathcal M\}
\end{equation}
Here, 
$$\log_+t=\max\{\log t,0\},\quad t>0.$$

Observing that 
\begin {equation*}
\mu(\log_+ \vert x\vert)=\log_+\mu(x),\quad x\in S(\tau),
\end{equation*}
it follows that, if $x\in S(\tau)$, then
\begin{align}
\label{eqnFK2}
x\in L_{\log_+}(\tau)\quad 
&\iff\quad  
\int_0^t \log_+\mu(x)dm<\infty,\quad \forall t>0\\
&\iff \quad \log_+\mu(x)\in L_1+L_\infty\\
&\iff \quad \int_0^1 \log_+\mu(x)dm<\infty.
\end{align}
It is clear that if $x\in L_{\log}(\tau)$, then $x^*, \vert x\vert, uxv \in L_{\log}(\tau)$ for all unitary $u,v\in {\mathcal M}$; moreover, if $x\geq 0$ and $\alpha>0$ then the equality $\mu(x^\alpha)=\mu(x)^\alpha$ readily implies that $x^\alpha\in L_{\log}(\tau)$. 

If $\tau({\bf 1})<\infty$, then it follows from \cite{DPS} Chapter III , remarks preceding Lemma 3.6, that

\begin{equation*}
\tau(\log_+\vert x\vert)
=\int_0^{\tau({\bf 1})}\log_+\mu(x)dm
=\int_0^{\tau({\bf 1})}\log_+sd(\tau e^{\vert x\vert})(s).
\end{equation*}
Consequently, if $\tau({\bf 1})<\infty$, then the class $L_{\log_+}(\tau)$ coincides with the class ${\mathcal M}^\Delta$ introduced by Haagerup and Schultz \cite{HS}. 

We now observe that the class
$L_{\log_+}(\tau)$ is an algebra. First recall that, if $s>0$ then the dilation operator $\sigma$ is defined on $S(m)$ by setting
$$
(\sigma_sx)(t) =x(t/s),\quad t>0.
$$
If $x,y\in S(\tau)$, then
\begin{equation}
\label{eqnDil}
\mu(x+y)\leq \sigma_2(\mu(x))+\sigma_2(\mu(y)),\quad \mu(xy)\leq \sigma_2(\mu(x))\sigma_2(\mu(y))
\end{equation}
See, for example, \cite{FaKo1986} Lemma 2.5. See also \cite{LSZ}.

\begin{prop}
\label{propFK}
 If $x,y\in L_{\log_+}(\tau)$, then $x+y,xy\in L_{\log_+}(\tau)$.
\end{prop}
\begin{proof}If $x,y\in  L_{\log_+}(\tau)$ then, using (\ref{eqnDil}), 
\begin{equation*}
\mu(x+y)\leq \sigma_2(\mu(x))+\sigma_2(\mu(y))\leq \max\{2\sigma_2(\mu(x)),2\sigma_2(\mu(y))\}.
\end{equation*}
Consequently,
\begin{align*}
\log_+(\mu(x+y))
&\leq \log_+ \left(\max\{2\sigma_2(\mu(x)),2\sigma_2(\mu(y))\}\right)\\
&=\max\{\log_+(2\sigma_2(\mu(x)),\log_+(2\sigma_2(\mu(y))\}\\
&\leq \log_+(2\sigma_2(\mu(x)))+\log_+(2\sigma_2(\mu(y)))\\
&=\sigma_2(\log_+(2\mu(x))+\log_+(2\mu(y)))\\
&\leq \sigma_2(\log_+\mu(x)+\log_+\mu(y)+2\log 2)
\end{align*}
Since $$
\log_+\mu(x),\log_+\mu(y)\in L_1+L_\infty,
$$
as follows from (\ref{eqnFK2}), it follows that
$$
\sigma_2(\log_+\mu(x)+\log_+\mu(y)+2\log 2)\in L_1+L_\infty.
$$
This implies that 
$$\log_+\mu(x+y)\in L_1+L_\infty,$$
and accordingly, $x+y\in L_{\log_+}(\tau)$.

To see that $xy\in L_{\log_+}(\tau)$, observe first, again using (\ref{eqnDil}), that
\begin{equation*}
\mu(xy)\leq \sigma_2(\mu(x)\mu(y))=\sigma_2(\mu(x))\sigma_2(\mu(y)).
\end{equation*}
Consequently, 
\begin{align*}
\log_+\mu(xy)
&\leq \log_+\left(\sigma_2(\mu(x))\sigma_2(\mu(y))\right)
=\sigma_2(\log_+\mu(x)+\log_+\mu(y))\\
&\in L_1+L_\infty,
\end{align*}
and this implies that $xy\in L_{\log_+}(\tau)$.

Finally, since $\mu(x)=\mu(x^*)$, the final assertion of the proposition is trivial.

\end{proof}

It should be noted that if $x\in S(\tau)$ satisfies a (so-called) "Lorentz space"-type condition (see \cite{FaKo1986})
\begin{equation}
\label{eqnFK3}
x\in {\mathcal M}\quad {\rm or}\quad \mu(s;x)\leq Cs^{-\alpha}\quad (C,\alpha>0),\ s>0
\end{equation}
then $x\in L_{\log_+}(\tau)$. Indeed, this assertion is clear if $x\in {\mathcal M}$. Assume then that $x\in S(\tau)$ and that $\mu(s;x)\leq C s^{-\alpha}$ for some $C,\alpha>0$ and all $s>0$ and for simplicity, suppose that $C=1$. Observe that
$$\int_0^1\log_+\mu(s;x)ds
\leq \int_0^1\log_+s^{-\alpha}ds=-\alpha\int_0^1\log sds<\infty
$$
 and this suffices to show that $x\in  L_{\log_+}(\tau)$. In particular, it follows that if $x\in L^p(\tau)$ for some $0<p<\infty$, then $x\in L_{\log_+}(\tau)$.
 Indeed, as noted in ~\cite{FaKo1986}, in this case
 $$
 \mu(s;x)^p\leq s^{-1}\int_0^s\mu(x)^pdm\leq s^{-1}\Vert x\Vert _p^p
 $$
 for all $s>0$. 
 
 On the other hand, if $x\in L_{\log_+}(\tau)$, then it need not be the case that $x$ satisfies any Lorentz space condition. By way of example, if 
 $f(t)=\exp(t^{-\frac{1}{2}}), t>0$, then it is evident that $f\in L_{\log_+}$. However, since
 \begin{equation*}
 \lim_{t\to 0}t^{-\alpha}\exp(-t^{-{\frac{1}{2}}})=0, \quad\forall \alpha>0,
 \end{equation*}
 it follows that $f$ fails to satisfy any condition of the type (\ref{eqnFK3}).
 
 One immediate consequence of the preceding is that 
 \begin{equation*}
 L_1(\tau)+{\mathcal M}\subseteq L_{\log_+}(\tau)\subseteq S(\tau).
 \end{equation*}
 Note that the inclusion on the left hand side is, in general, proper by observing that, if  $f(t)= 1/t, t>0$, then $f\in L_{\log_+}$ but  
 $\mu(f)\not\in L_1 + L_\infty$.

 The proposition which follows is due to L.G.Brown \cite{Br} Proposition 1.11. It is an extension of the Weyl majorization theorem given in ~\cite{Fa} Corollary 4.2 which goes back to the paper of Weyl\cite{W}.
 
 \begin{prop}
 \label{propBr}
 Suppose that $f,g:(0,\infty)\to[0,\infty)$ be non-increasing.
 If $f,g\in L_{log_+}$ then the following statements are equivalent:
 \begin{align*}
 (i)\quad \int_0^t\log f(s)ds&\leq \int_0^t \log g(s)ds,\quad \forall t>0.\\
 (ii)\quad \int_0^\infty\log _+(rf(s))ds&\leq \int_0^\infty \log_+ (rg(s))ds,\quad \forall r>0,\quad \forall t>0.\\
 (iii)\quad \int_0^t\varphi(f(s))ds&\leq \int_0^t\varphi(g(s))ds,\quad \forall t>0,
 \end{align*}
 \quad \quad for all continuous increasing functions $\varphi:[0,\infty)\to [0,\infty)$ such that $\varphi(0)=0$ and $\varphi\circ\exp$ is convex.
\end{prop}

\begin{rem}
\label{rmkLS}It should be observed that a simple consequence of Proposition~\ref{propBr} preceding is that if $x,y\in L_{\log_+}(\tau)$ and if $x$ is logarithmically submajorised by $y$ in the sense that
\begin{equation*}
\int_0^t\log \mu(s;x)ds\leq \int_0^t\log \mu(s;y)ds,\quad \forall t>0,
\end{equation*} 
then $x$ is submajorised by $y$ in the usual sense, that is,
\begin{equation*}
\int_0^t\mu(s;x)ds\leq \int_0^t \mu(s;y)ds \quad \forall t>0.
\end{equation*}
 Indeed, one need only take $\varphi(t)=t, t\geq 0$.
\end{rem}

\section{The determinant function}

If $x\in L_{\log_+}(\tau)$ 
then the {\it determinant function} $\Lambda(x)$ is defined by setting, for all $t>0$,

$$
\Lambda(t;x)=\exp \left(\int_0^t \log \mu(s;x)\, ds \right).
$$
Since $x\in L_{\log_+}(\tau)$, it follows that the integral is well-defined in the sense that it is either finite or takes the value $-\infty$. Indeed, 
$$
\Lambda(t;x)=0\Leftrightarrow \int_0^t \log \mu(s;x)ds=-\infty
$$
In the case that $x$ is $\tau$-compact, this function, which goes back to A.Grothendieck~\cite{G}, was studied by T.Fack~\cite{Fa},~\cite{Fa83} and extended further in~\cite{FaKo1986} to the class of  operators satisfying a "Lorentz space"-type condition.

We remark that, if $\tau({\bf 1})<\infty$, then $\Delta(x):=\Lambda(\tau({\bf 1});x)$ is precisely the extension of the classical Fuglede-Kadison determinant~\cite{FK} to $S(\tau)$ given by Haagerup and Schultz~\cite{HS} who proved  the following
fundamental result.
\begin{prop}
\label{propHS} If $\tau({\bf 1})<\infty$ and if $x,y\in L_{\log_+}(\tau)$, then 
$$
\Delta(xy)=\Delta(x)\Delta(y).
$$

\end{prop}

Let us note that if $\M=B(H)$ with standard trace and if $x\in\M$ is
compact then
$$
\Lambda(n+1;x)=\prod_{k=0}^n \mu(k;x).
$$
We note first  some elementary properties of the
determinant function $\Lambda$ which follow directly from the
definition and properties of the singular value function (see ~\cite{FaKo1986} Lemma 2.5) together and the definition (\ref{propFK}).

\begin{enumerate}
\item If $0\le x\in L_{\log_+}(\tau)$ and $\alpha>0$, then
$x^\alpha\in L_{\log_+}(\tau)$ and
$
\Lambda(x^\alpha)=\Lambda(x)^\alpha.
$
\item If $x\in L_{\log_+}(\tau)$, then $x^*,\vert x\vert, x^*x\in L_{\log_+}(\tau)$ and 
$$\Lambda(x)=\Lambda(|x|)=\Lambda(x^*)=\Lambda(x^*x)^{1/2}.$$
\item If $x\in L_{\log_+}(\tau)$, then 
$$\Lambda(x^* x)=\Lambda(x x^*).$$
\item If $x,y\in L_{\log_+}(\tau)$ then
$$\Lambda(xy)=\Lambda(|x||y^*|).$$
\item If $x\in L_{\log_+}(\tau)$, then $uxv\in L_{\log_+}(\tau)$
for all unitary $u,v\in{\mathcal M}$, in which case
$$\Lambda(uxv)=\Lambda(x).$$ 
\end{enumerate}

For example, to indicate the proof of 4. above, it suffices to show that 
\begin{equation*}
\mu(xy)=\mu(\vert x\vert \vert y^*\vert).
\end{equation*}
To this end,
let $x=u\vert x\vert $ be the polar decomposition, and observe that 
\begin{align*}
\mu(xy)&=\mu(u\vert x\vert y)\leq \mu(\vert x\vert  y)\\
&=\mu(y^*\vert x\vert)\leq \mu(\vert y^*\vert\vert x\vert)=
\mu( \vert x\vert\vert y^*\vert)\\
&=\mu(u^*  x \vert y^*\vert )\leq \mu(x\vert y^*\vert)\\
&=\mu(\vert y^*\vert x^*)=\mu(y^*x^*)=\mu(xy),
\end{align*}
using repeatedly ~\cite{FaKo1986} Lemma 2.5 (vi).
\bigskip

The Theorem which follows is the principal result of this section and is a refinement of ~\cite{FaKo1986} Theorem 4.2 (ii). See also \cite{G}, Th\'eor\'eme 1, where it is stated for the case  that $a,b$ are compact operators in some Hilbert space.
\begin{thm}
\label{thmFK}{\bf (Weyl inequality)}
If $x,y\in L_{\log_+}(\tau)$, then 
\begin{equation}
\label{eqnWI}
\Lambda(t;xy)\leq \Lambda(t;x)\Lambda(t;y),\quad \forall t>0.
\end{equation}
\end{thm} 
 We shall need the following geometric characterisation of the singular value function:  If $x\in S(\tau)$ and if $t>0$, then
\begin{equation}
\label{eqnFK4}
\mu(t;x)=\inf\{\Vert x-z\Vert _\infty:\tau(s(\vert z\vert))\leq t\}. 
\end{equation}
See, for example, \cite{FaKo1986} Proposition 2.4.

We prove first the following result.
\begin{lem}
\label{lemFK1}If $a,b\in L_{\log_+}(\tau)$ are such that $a,b\geq {\bf 1}$, if $d(1;a),d(1;b)<\infty$ and if $r=e^a(1,\infty)\vee e^b(1,\infty)$, then 

$$
\Lambda(t;ab)= \Lambda(t;a)\Lambda(t;b),\quad \forall t\geq \tau(r).
$$
\end{lem}

\begin{proof}
We set 
\begin{equation*}
p=\chi_{(1,\infty)}(a)=e^a(1,\infty),\quad q=\chi_{(1,\infty)}(b)=e^b(1,\infty),\quad r=p\vee q,
\end{equation*}
and note that 
\begin{equation*}
\tau(r)\leq \tau(p)+\tau(q)=d(1;a)+d(1;b)<\infty.
\end{equation*}
Since $a,b\geq 1$, it follows that
\begin{equation*}
{\bf 1}-p=\chi_{[0,1]}(a)=\chi_{\{1\}}(a),\quad {\bf 1}-q=\chi_{\{1\}}(b)
\end{equation*} 
and so 
\begin{equation*}
{\bf 1}-r=({\bf 1}-p)\wedge ({\bf 1}-q)=\chi_{\{1\}}(a)\wedge \chi_{\{1\}}(b).
\end{equation*}
Consequently,
\begin{equation}
\label{eqnFKa}
a({\bf 1}-r)=a\chi_{\{1\}}(a)({\bf 1}-r)=\chi_{\{1\}}(a)({\bf 1}-r)
=({\bf 1}-r)
\end{equation}
and
\begin{equation}
\label{eqnFKb}
({\bf 1}-r)a=(a({\bf 1}-r))^*=({\bf 1}-r).
\end{equation}
Similar reasoning applies with $a$ replaced by $b$ so that also
\begin{equation}
\label{eqnFKc}
({\bf 1}-r)b=(b({\bf 1}-r))^*=({\bf 1}-r)
\end{equation}
and it follows, in particular, that $a,b$ commute with $r$. 

Using (\ref{eqnFKa}), (\ref{eqnFKc}), note further that
\begin{equation}
\label{eqnFKe}
ab=rab+({\bf 1}-r)ab=rab+({\bf 1}-r). 
\end{equation}
It will now be shown that 
\begin{equation}
\label{eqnFKcc}
\mu(t;a)=\mu(t;rar),\quad \mu(t;b)=\mu(t;rbr),\quad t\in (0,\tau(r)).
\end{equation}
Without loss of generality, it may be assumed that $\tau(r)=1$. 
By  \cite{DDP1989a} Lemma 2.4 (i), it follows that 
\begin{align*}
\mu(t;a)&=\mu(t;pap)=\mu(t;p\cdot rar\cdot p)\\
&\leq \mu(t;rar)\leq \mu(t;a),\quad t\in (0,\tau(p)),
\end{align*}
so that
\begin{equation*}
\mu(t;rar)
=\mu(t;a),\quad t\in (0,\tau(p)).
\end{equation*}
Similarly,
\begin{equation*}
\mu(t;rbr)
=\mu(t;b),\quad t\in (0,\tau(q)).
\end{equation*}
If $\tau(p)=1$, then the first equality in (\ref{eqnFKcc}) follows immediately, and if  $\tau(q)=1$, then the second equality in (\ref{eqnFKcc}) holds. Consequently, it may be assumed that $\tau(p)<1$ and $\tau(q)<1$. 

Suppose now that $\tau(p)\leq t<1$. Observing that 
$\tau(s(pa))\leq \tau(p)\leq t$, it follows from (\ref{eqnFK4}) that 
\begin{equation*}
\mu(t;a)\leq \Vert a-pa\Vert _\infty=\Vert a\chi_{[0,1]}(a)\Vert _\infty \leq 1.
\end{equation*}
Since $a\geq {\bf 1}$, it follows that $\mu(t;a)\geq 1$ for all $t>0$ and it now follows that 

\begin{equation}
\label{eqnFKkka}
\mu(t;a)=1\quad \forall t\geq \tau(p). 
\end{equation}
The same argument shows that

\begin{equation}
\label{eqnFKkkb}
\mu(t;b)=1\quad \forall t\geq \tau(q). 
\end{equation}
In particular, $\mu(t;a)=1$ for all $\tau(p)\leq t<1$ and 
$\mu(t;b)=1$ for all $\tau (q)\leq t<1$.  

At the same time, since $a$ commutes with $r$, it follows that  $a\geq rar\geq r$ since $a\geq 1$. This implies that 
\begin{equation*}\mu(a)\geq \mu(rar)\geq \mu(r)=\chi_{[0,\tau(r))}.
\end{equation*}
 Consequently,
\begin{equation*}
1=\mu(t;a)\geq \mu(t;rar)\geq 1,\quad \tau(p)\leq t<\tau(r)=1
\end{equation*}
so that 
\begin{equation*}
\mu(t;rar)=1=\mu(t;a), \quad \tau(p)\leq t<1.
\end{equation*}
Similarly,
\begin{equation*}
\mu(t;rbr)=1=\mu(t;a), \quad \tau(q)\leq t<1.
\end{equation*}
This suffices to complete the proof of the equalities in (\ref{eqnFKcc})

It will now be shown that 
\begin{equation}
\label {eqnFKf}
\mu(t;ab)=\mu(t;rabr)=\mu(t;rar\cdot rbr),\quad 0<t<1.
\end{equation}
Note first that
\begin{align}
\begin{split}
\label{eqnFKg}
\mu(rabr)=&\mu(rar\cdot rbr)=\mu^{\frac{1}{2}}(rbr(rar)^2rbr)
\\
&\geq \mu^{\frac{1}{2}}(rbr\cdot r\cdot rbr)=\mu(rbr)\geq \mu(r)=\chi_{[0,1)}
\end{split}
\end{align}

Observe next that $r$ commutes with $\vert rabr\vert $, since it is clear that $r$ commutes with $\vert rabr\vert^2=rb^*a^*rabr=rba^2br$. From this it follows that
\begin{equation}
\label{eqnFKgg}
r\vert rabr\vert r =\vert rabr\vert.
\end{equation}
Indeed, 
\begin{align*}
(r\vert rabr\vert r)^2
&=r\vert rabr\vert r\vert rabr\vert r
=r\vert rabr\vert ^2r\\
&=rb^*a^*rabr=rb^*a^*abr=\vert rabr\vert ^2
\end{align*}
and the equality (\ref {eqnFKgg}) follows by passing to square roots. We may now show that 
\begin{equation}
\label{eqnFKee}
\vert ab\vert =r\vert rabr\vert r +({\bf 1}-r)
=\vert rabr\vert +({\bf 1}-r).
\end{equation}
In fact, using (\ref{eqnFKe}) and (\ref{eqnFKgg}), observe that
\begin{align*}
\vert ab\vert ^2
&=(ab)^*(ab)=rb^*a^*rabr +({\bf 1}-r)\\
&=\vert rabr\vert ^2 +({\bf 1}-r)
=(r\vert rabr\vert r)^2+({\bf 1}-r)\\
&=(r\vert rabr\vert r +({\bf 1}-r))^2.
\end{align*}
Passing to square roots and again using (\ref{eqnFKgg}) now  yields
\begin{align*}
\vert ab\vert 
=r\vert rabr\vert r +({\bf 1}-r)
=\vert rabr\vert+({\bf 1}-r).
\end{align*}

Observe now, that if $s>0$, then, using (\ref{eqnFKee})
\begin{align*}
d(s;ab)&=d(s;\vert rabr\vert+({\bf 1}-r))\\
&=\begin{cases}
d(s;\vert rabr\vert),&\ s\geq 1,\\
d(s;\vert rabr\vert)+d(s;({\bf 1}-r)),&\ 0<s<1.
\end{cases}\\
&=
\begin{cases}
d(s;\vert rabr\vert),&\ s\geq 1,\\
\infty,&\ 0<s<1.
\end{cases}
\end{align*}

It follows readily that
\begin{align*}
\mu(t;ab)&=\inf\{s\geq 0:d(s;ab)\leq t\}\\
&=\mu(rabr)\chi_{[0,1)}+\chi_{[1,\infty)},
\end{align*}
and this yields (\ref{eqnFKf}). In particular, it follows also that
\begin{equation}
\label{eqnFKk}
\mu(t;ab)=1,\quad t\geq 1.
\end{equation}

Finally, if $t\geq \tau(r)=1$, then
\begin{align*}
\Lambda(t;ab)
&=\exp\int_0^t\log\mu(s;ab)ds\overset{\ref{eqnFKk}}{=}
\exp\int_0^1\log\mu(s;ab)ds\\
&
=\overset{(\ref{eqnFKf})}{=}\exp\int_0^{1}\log\mu(a;rabr)ds\\
&
\overset{Prop\ref{propHS}}{=}
\left(\exp\int_0^{1}\log\mu(s;rar)ds\right)\left(\exp\int_0^{1}\log\mu(s;rbr)ds\right)\\
&\overset{(\ref{eqnFKcc})}{=}
\left(\exp\int_0^{1}\log\mu(s;a)dt\right)\left(\exp\int_0^{1}\log\mu(s;b)ds
\right)\\
&\overset{(\ref{eqnFKkka}),(\ref{eqnFKkkb})}{=}
\left(\exp\int_0^{t}\log\mu(s;a)dt\right)\left(\exp\int_0^{t}\log\mu(s;b)ds
\right)
\end{align*}
and this suffices to complete the proof of the Lemma.

\end{proof}

We may now prove the following special case of Theorem~\ref{thmFK}.
\begin{lem}
\label{lemFK2}
If $a,b\in L_{\log_+}(\tau)$ satisfy $\mu(s;a)>0,\mu(s;b)>0$ for every $s>0$
then
$$
\Lambda(t;ab)\leq \Lambda(t;a)\Lambda(t;b),\quad \forall t>0.
$$
\end{lem}
\begin{proof}
Without loss of generality, it may be assumed that $a,b\geq 0$. Let $t>0$ be given. It will be shown that 
\begin{equation}
\label{eqnFKaa}
\int_0^t\log \mu(s;ab)ds\leq \int_0^t\log \mu(s;a)ds+\int_0^t\log \mu(s;b)ds
\end{equation}
Since $\mu(t;a)>0$ and $\mu(t;b)>0$, it may be assumed further, by homogeneity, that $\mu(t;a)=1=\mu(t;b)$. Now set 
\begin{equation}
\label{eqnFKaaa}
x=a\vee {\bf 1},\quad y=b\vee {\bf 1}
\end{equation}
and observe that 
\begin{equation}
\label{eqnFKac}
\mu(s;x)=\mu(s;a)\vee 1=
\begin{cases}
\mu(s;a),&\ s\in (0,t),\\
1,&\ s\in(t,\infty).
\end{cases}
\end{equation}
and
\begin{equation}
\label{eqnFKad}
\mu(s;y)=\mu(s;b)\vee 1=
\begin{cases}
\mu(s;b),&\ s\in (0,t),\\
1,&\ s\in(t,\infty).
\end{cases}
\end{equation}
Now observe that
\begin{equation}
\label{eqnKaab}
\mu(ab)\leq \mu(xy)\quad {\rm and}\quad \mu(xy)\geq 1.
\end{equation}
To see the first inequality, using the fact that $a\geq 0,b\geq 0$,
observe that
\begin{align*}
\mu(ab)
&=\mu(ba^2b)^{\frac{1}{2}}\leq \mu(b(a\vee {\bf 1})^2b)^{\frac{1}{2}}\\
&=\mu((a\vee {\bf 1})b)=\mu(((a\vee {\bf 1})b)^*)=\mu(b(a\vee {\bf 1}))\\
&\leq \mu((b\vee {\bf 1})(a\vee {\bf 1}))
=\mu(((b\vee {\bf 1})(a\vee {\bf 1}))^*)=\mu((a\vee {\bf 1})(b\vee {\bf 1})).
\end{align*}
The second inequality follows by observing that $x^2\geq {\bf 1},\ y^2\geq {\bf 1}$ and
\begin{align*}
\mu(xy)&=\mu(yx^2y)^{\frac{1}{2}}\geq \mu(y^2)^{\frac{1}{2}}\geq 1.
\end{align*}

If $t\leq w=\tau(e^x(1,\infty)\vee e^y(1,\infty))$,  observe that
\begin{align*}
\int_0^t\log \mu(s;ab)ds
&\overset{(\ref{eqnKaab})}{\leq} \int_0^t\log \mu(s;xy)ds 
\overset{}{\leq} \int_0^w\log\mu(s;xy)ds\\
&
\overset{Lemma \ref{lemFK1}}{=}\int_0^w\log\mu(s;x)ds
+\int_0^w\log\mu(s;y)ds\\
&
\overset{(\ref{eqnFKac}),(\ref{eqnFKad})}{=}\int_0^t\log\mu(s;x)ds
+\int_0^t\log\mu(s;y)ds\\
&=\int_0^t\log\mu(s;a)ds
+\int_0^t\log\mu(s;a)ds.
\end{align*}

\end{proof}

Similarly, if $t\geq \tau(e^x(1,\infty)\vee e^y(1,\infty))$,

\begin{align*}
\int_0^t\log \mu(s;ab)ds
&\overset{(\ref{eqnKaab})}{\leq} \int_0^t\log \mu(s;xy)ds 
\\
&
\overset{Lemma \ref{lemFK1}}{=}\int_0^t\log\mu(s;x)ds
+\int_0^t\log\mu(s;y)ds\\
&=\int_0^t\log\mu(s;a)ds
+\int_0^t\log\mu(s;a)ds.
\end{align*}
{\bf Proof of Theorem~\ref{thmFK}}: Assume that $x,y\in L_{\log_+}(\tau)$. Without loss of generality, it may be assumed that $x,y\geq 0$. Let $t>0$. We need to establish the equality
\begin{equation}
\label{eqnFKz}
\int_0^t\log\mu(s;xy)ds
\leq \int_0^t\log\mu(s;x)ds+\int_0^t\log\mu(s;y)ds.
\end{equation}
For each $n=1,2,\dots$, set $x_n=x+\frac{1}{n}{\bf 1},\ y_n=
y+\frac{1}{n}{\bf 1}$ and note that
\begin{align*}
\mu(xy)&=\mu(yx^2y)^{\frac{1}{2}}\leq \mu(yx_n^2y)^{\frac{1}{2}}\\
&=\mu(x_ny)=\mu((x_ny)^*)=\mu(yx_n)=\mu(x_ny^2x_n)^{\frac{1}{2}}\\
&\leq \mu(x_ny_n^2x_n)^{\frac{1}{2}}=\mu(y_nx_n)=\mu(x_ny_n),\quad n\geq 1.
\end{align*}
By Lemma~\ref{lemFK2}, we now  have that
\begin{equation*}
\int_0^t\log \mu(s;xy)ds
\leq \int_0^t\log \mu(s;x_ny_n)ds
\leq \int_0^t\log \mu(s;x_n)ds+
\int_0^t\log \mu(s;y_n)ds.
\end{equation*}
By the monotone convergence theorem,
\begin{equation*}
\int_0^t\log \mu(s;x_n)ds\downarrow_n\int_0^t\log \mu(s;x)ds,\quad
\int_0^t\log \mu(s;y_n)ds\downarrow_n\int_0^t\log \mu(s;y)ds
\end{equation*}
and the inequality~(\ref{eqnFKz}) follows directly.

\section{ H\"older type inequalities via logarithmic submajorization}

The principal result of this section (Proposition~\ref{propKIp4}) is a very general H\"older inequality which goes back to Kosaki~\cite{Ko1998} in the special case of trace ideals. Our approach is based on logarithmic submajorization and while we follow the ideas of ~\cite{Ko1998}, the methods given there fail to apply in the present setting. 

We shall need the following submajorization inequalities
~\cite{Ko1992},~\cite{Ha}.

\begin{prop}
\label{propalt1}
Let $\varphi$ be a continuous increasing function on $[0,\infty)$ such that 
$\varphi(0)=0$ and has the property that $\varphi\circ\exp(\cdot)$ is convex.  If 
$0\leq a,b\in S(\tau)$,  then
\begin{equation}
\label{eqnKo5}
\int_0^t \varphi(\mu(\vert ab\vert ^r)dm
\leq \int _0^t\varphi(\mu(a^rb^r))dm,\quad r\geq 1\quad t>0.
\end{equation}
and
\begin{equation}
\label{eqnKo6}
\int_0^t \varphi(\mu(\vert ab\vert ^r)dm
\geq \int _0^t\varphi(\mu(a^rb^r))dm,\quad 0<r\leq 1\quad t>0.
\end{equation}

\end{prop}

The preceding  Araki-Lieb-Thirring type inequalities were first proved by 
Kosaki(~\cite{Ko1992} Theorem 2) for $\tau$-measurable operators $0\leq a,b$
which are $\tau$-compact, i.e., the operators $a,b$ satisfy the 
additional
assumption that $\lim_{t\to \infty}\mu(t;a)=0=\lim_{t\to \infty}\mu(t;b)$.
It has been noted recently by Han~\cite{Ha} Proposition 2.4 that Kosaki's 
arguments may be extended to
the general case.

We note the following consequence which shows that the function $t\to 
\vert a^tb^t\vert ^{\frac{1}{t}}$ is increasing with respect to submajorization,
 and is the counterpart for singular-value submajorization
of a result of Wang and Gong~\cite{WG} for positive semi-definite  matrices.
See also ~\cite{BHDA1995} Theorem 3. 

\begin{cor}
\label{corWG}
If $0\leq a,b\in S(\tau)$ and if $0<t\leq u<\infty$,
then
\begin{equation}
\label{eqnWG}
\vert a^tb^t\vert ^{\frac{1}{t}}
\prec\prec
\vert a^ub^u\vert ^{\frac{1}{u}}.
\end{equation}

\end{cor} 
\begin{proof}
In equation~(\ref{eqnKo6}), set $r=t/u,\ \varphi(\cdot)=(\cdot)^{\frac{1}{t}}$ to obtain
\begin{equation*}
\mu(\vert a^{\frac{t}{u}}b^{\frac{t}{u}}\vert ^{\frac{1}{t}})=
\left[ \mu(a^{\frac{t}{u}}b^{\frac{t}{u}})\right]^{\frac{1}{t}}
\prec\prec \left[\mu(\vert ab\vert ^{\frac{t}{u}})\right]^{\frac{1}{t}}
=\mu(\vert ab\vert ^{\frac{1}{u}}).
\end{equation*}
The submajorisation (\ref{eqnWG}) now follows by replacing $a, b$ by $a^u,b^u$ respectively.

\end{proof}

Before proceeding, it is desirable to introduce some additional notation. If $x,y\in L_{\log_+}(\tau)$, then we shall say that $x$ is {\it logarithmically submajorized} by $y$, written $x\prec\prec_{\log}y$, if and only if 
\begin{equation*}
\int_0^t\log \mu(s;x)ds\leq \int_0^t\log \mu(s;y)ds
\end{equation*}
Observe that $x\prec\prec_{\log}y$ if and only if $\Lambda(x)\leq \Lambda(y)$. With this notation, it should also be noted, for ease of reference,  that the Weyl inequality 
\begin{equation*}
\Lambda(xy)\leq \Lambda(x)\Lambda(y)
\end{equation*}
given in Theorem~\ref{thmFK}(ii) may be reformulated in terms of logarithmic submajorisation by the inequality
\begin{equation*}
xy\prec\prec_{\log}\mu(x)\mu(y).
\end{equation*}
 Indeed, one need only observe that, for all $t>0$,
\begin{align*}
\int_0^t\log\mu(s;xy)ds
&=\log \Lambda(t;xy)\\
&\leq \log \Lambda(t;x)+\log \Lambda(t;y)\\
&=\int_0^t\log \mu(s;x)ds +\int_0^t\log\mu(s;y)ds\\
&=\int_0^t\log(\mu(s;x)\mu(s;y))ds.
\end{align*}

\begin{cor}
\label{corKALT} If $0\leq a,b\in L_{\log_+}(\tau)$ 
then 
\begin {equation}
\label{eqnA9}
\Lambda(\vert ab\vert ^r)\leq \Lambda (a^rb^r),\quad r\geq 1;
\end {equation}
and
\begin{equation}
\label{eqnA10}
\Lambda(\vert ab\vert ^r)\geq \Lambda(a^rb^r),\quad 0<r\leq 1.
\end{equation}
\end{cor}

The Corollary follows immediately from Proposition~\ref{propalt1} via the implication (iii)$\Longrightarrow$ (i) of Proposition~\ref{propBr}.

We need
the following generalization of the  polar decomposition given 
in ~\cite{GMMN} Theorem 2.7

\begin{lem}
\label{lemGMMN} If $x\in S(\tau)$ has polar decomposition $x=u\vert x\vert$ and
if $\phi,\psi$ are Borel functions on ${\mathbb R}$ such that
$\phi(\lambda)\psi(\lambda)=\lambda,\ \lambda \in {\mathbb R}$ then
\begin{equation}
\label{eqnBMMN}
x=\phi(\vert x^*\vert)u\psi(\vert  x \vert).
\end{equation}
In particular,
if $1/p+1/q=1$, then
$$
x=\vert x^*\vert ^{\frac{1}{p}} u  \vert x\vert ^{\frac{1}{q}}.
$$
\end{lem}
\begin{prop}
\label{propKIp3}
If $a,b,x\in L_{\log_+}(\tau)$ with $a,b\geq 0$, if $1/p+1/q=1$ and if $r>0$ then
$$
\Lambda (\vert axb\vert ^r)\leq \Lambda^{\frac{r}{p}}(a^p x)\Lambda^{\frac{r}{q}}(x b^q),
$$
or, equivalently,
\begin{equation}
\label{eqnKIp3}
\vert axb\vert ^r\prec\prec_{\log} \mu(\vert a^px\vert^{\frac{r}{p}})\mu(\vert xb^q\vert ^{\frac{r}{q}}).
\end{equation}

\begin{proof}
If $x=u\vert x\vert$ is the polar decomposition, then it follows from Lemma
~\ref{lemGMMN} that 
$$
axb=a\vert x^*\vert ^{\frac{1}{p}}u\vert x\vert ^{\frac{1}{q}}b.
$$

It follows from the inequality (\ref{eqnWI})
that

$$
\Lambda (axb)\leq \Lambda (a\vert x^*\vert ^{\frac{1}{p}})
\Lambda (u\vert x\vert ^{\frac{1}{q}}b)
\leq
\Lambda (a\vert x^*\vert ^{\frac{1}{p}})
\Lambda (\vert x\vert ^{\frac{1}{q}}b).
$$
Using Corollary~\ref{corKALT}, observe that
\begin{align*}
\Lambda (a\vert x^*\vert ^{\frac{1}{p}})&\leq \Lambda^{\frac{1}{p}}(a^p\vert x^*\vert)
=\Lambda^{\frac{1}{p}}(\vert x^*\vert a^p)= \Lambda^{\frac{1}{p}}( x^* a^p)=
\Lambda^{\frac{1}{p}}(  a^px)\\
\Lambda(\vert x\vert ^{\frac{1}{q}}b)
&\leq \Lambda^{\frac{1}{q}}(\vert x\vert b^q)=\Lambda^{\frac{1}{q}}( x b^q)
\end{align*}
Consequently,
$$
\Lambda (axb)\leq \Lambda^{\frac{1}{p}}(a^p x)\Lambda^{\frac{1}{q}}(x b^q),
$$
and the assertion of the Proposition now follows by observing that 
\begin{equation*}
\Lambda (axb)^r=\Lambda (\vert axb\vert ^r).
\end{equation*}

\end{proof}
\end{prop}

The following generalised H\"older inequality for normed order ideals is 
well-known.  The simple proof is included for the sake of completeness.
\begin{prop}
\label{propHo}
Suppose that $E\subseteq S(m)$ is a a normed order ideal.
Suppose that $r,p,q>0$ and that $1/r=1/p+1/q$. If $\vert x\vert ^p\in E,
\vert  y\vert ^q\in 
E$, then
$\vert xy\vert ^r\in E$ and 
$$
\Vert \vert xy\vert^r \Vert _{E}^{\frac{1}{r}}\leq 
\Vert \vert x\vert ^p\Vert _{E}^{\frac{1}{p}}\Vert \vert y\vert ^q
\Vert _{E}^{\frac{1}{q}}.
$$
\end{prop}
\begin{proof}
It suffices to assume that $\Vert x\Vert _{E^{p}}=1=\Vert y\Vert _{E^{q}}$.
Using the well-known numerical Young's inequality, it follows that
$$
\frac{\vert x\vert^r\vert y\vert^r}{r}\leq 
\frac{\vert x\vert^{p}}{p}+\frac{\vert y\vert ^{q}}{q}.
$$
Since $ \vert x\vert^{p},\vert y\vert ^{q}\in E$, it follows from the
linearity of $E$ and the fact that $E$ is an order ideal, that 
$\vert x\vert^r\vert y\vert^r\in E$. Consequently, using
the fact that $E$ is a normed ideal,
$$
\Vert \ \vert xy\vert^r\ \Vert_E
\leq r\left(
\frac{\Vert\ \vert x\vert^{p}\ \Vert_E}{p}
+\frac{\Vert\ \vert y\vert^{q}\ \Vert_E}{q}
\right)\leq 1.
$$
This  suffices to
complete the proof.
\end{proof}

If $E\subseteq S(\tau)$ is a symmetric (Banach) space, then $E\subseteq L_1(\tau) +{\mathcal M}$ so that $E\subseteq L_{\log_+}(\tau)$.  The norm on the symmetric space $E$ is said to be {\it monotone with respect to logarithmic submajorization} if whenever $x\in E, y\in L_{\log_+}(\tau)$ satisfy
$y\prec\prec_{\log}x$, it follows that $y\in E$ and $\Vert y\Vert _E\leq \Vert x\Vert _E$.

The principal result of this section now follows. In the setting of trace ideals, this result may be found in \cite{Ko1998} Theorem 3; however, the proof given in 
\cite{Ko1998} uses a standard trick with anti-symmetric tensors which is not available in the present more general setting.

\begin{prop}
\label{propKIp4} Suppose that  $E$ is a  symmetric space on $[0,\infty)$ whose norm is  
monotone with respect to logarithmic submajorisation. Suppose that $x\in
S(\tau)$, that $0\leq a,b\in S(\tau)$ and that $1<p,p^\prime<\infty$ with
conjugate exponents $q=p/(p-1),\ q^\prime =p^\prime/(p^\prime-1)$. If $r>0$ and
if $\vert
a^px\vert^{\frac{rp^\prime}{p}},\ \vert xb^q\vert^{\frac{rq^\prime}{q}}\in
E(\tau)$
then $\vert axb\vert ^r\in E(\tau)$ and
\begin{equation}
\label{eqnHoIneq}
\Vert \ \vert axb\vert ^r\Vert _{E(\tau)}
\leq 
\Vert \ \vert
a^px\vert^{\frac{rp^\prime}{p}}\Vert _{E(\tau)}^{\frac{1}{p^\prime}}
\Vert \ \vert xb^q\vert^{\frac{rq^\prime}{q}}\Vert
_{E(\tau)}^{\frac{1}{q^\prime}}.
\end{equation}
\end{prop}
\begin{proof}
Using the submajorization
$$
\vert axb\vert^r
\prec\prec_{\log}
\mu(\vert a^px\vert ^{\frac{r}{p}})
\mu(\vert xb^q\vert ^{\frac{r}{q}})
$$
given by (\ref{eqnKIp3}),
and observing that 
\begin{equation*}
\frac{1}{(p^\prime r)}+\frac{1}{(q^\prime r)}=\frac{1}{r},
\end{equation*} 
it then follows from the first assertion of Proposition~\ref{propHo} and the assumptions  

\begin{align*}\mu(\vert a^px\vert^{\frac{rp^\prime}{p}})
&=\mu(\vert a^px\vert ^{\frac{1}{p}})^{rp^\prime}\in E\\
\mu(\vert xb^q\vert^{\frac{rq^\prime}{q}})
&=\mu(\vert xb^q\vert^{\frac{1}{q}})^{rq^\prime}\in
E
\end{align*} 
 that 
\begin{equation*}
\left(\mu(\vert a^px\vert ^{\frac{1}{p}})\mu(\vert xb^q\vert ^{\frac{1}{q}})\right)^r
=\mu(\vert a^px\vert ^{\frac{r}{p}})
\mu(\vert xb^q\vert ^{\frac{r}{q}})\in E.
\end{equation*}
Since the norm on $E$ is monotone with respect to logarithmic submajorisation, it follows further that $\vert axb\vert ^r\in E(\tau)$
and
\begin{equation*}
\Vert \ \vert axb\vert ^r\Vert _{E(\tau)}
\leq \ \Vert 
\mu(\vert a^px\vert ^{\frac{r}{p}})
\mu(\vert xb^q\vert ^{\frac{r}{q}})\Vert _{E(\tau)}.
\end{equation*}
Now using the second assertion of  Proposition~\ref{propHo}, it follows that
\begin{align*}
\Vert 
\mu(\vert a^px\vert ^{\frac{r}{p}})
\mu(\vert xb^q\vert ^{\frac{r}{q}})\Vert _{E(\tau)}
&\leq
\Vert\vert
 a^px\vert^{\frac{rp^\prime}{p}}\Vert _{E(\tau)}^{\frac{1}{p^\prime}}
\Vert \ \vert xb^q\vert^{\frac{rq^\prime}{q}}\Vert
_{E(\tau)}^{\frac{1}{q^\prime}},
\end{align*}
and this suffices to complete the proof of the inequality (\ref{eqnHoIneq}).
\end{proof}

Some special cases are worth noting explicitly.

\begin{cor}
\label{corhiai} \textbf {(Hiai~\cite{Hi1997} p.174 ; ~\cite{Al}, Theorem 7))}
Let $E$ be a  symmetric space on $[0,\infty)$ whose norm is monotone with respect to logarithmic submajorisation and suppose that $a,b,x\in
S(\tau)$, that  $r>0$ and that $p,q>1$ satisfy $1/p+1/q=1$. If 
$\vert a^*ax\vert^{\frac{rp}{2}}, \ \vert xbb^*\vert^{\frac{rq}{2}}
\in E(\tau)$ then $\vert axb\vert ^r\in E(\tau)$ and
\begin{equation}
\label{eqnHiai}
\Vert \ \vert axb\vert^r\Vert_{E(\tau)}
\leq  \Vert \ \vert a^*ax\vert ^{\frac{rp}{2}}\Vert_{E(\tau)}^{\frac{1}{p}}
\Vert \ \vert xbb^*\vert ^{\frac{rq}{2}}\Vert_{E(\tau)}^{\frac{1}{q}}.
\end{equation}
\end{cor}
The Corollary follows immediately from Proposition~\ref{propKIp4} by taking $p=q=2$, then replacing $p^\prime,q^\prime$ by $p,q$ respectively
 and noting that
$$
\mu(\vert axb\vert ^r)=\mu^r(axb)=\mu^r(\vert a\vert x\vert b^*\vert )
=\mu(\vert \ \vert a\vert  x\vert b^*\vert\ \vert ^r).
$$

As noted in ~\cite{Hi1997}, two further special cases of interest are the
inequalities obtained by taking obtained $x={\bf 1}$, and $p=2$ 
respectively:
\begin{cor}
\label{corHiai3}
Let $E$ be a symmetric space on $[0,\infty)$ whose norm is monotone with respect to logarithmic submajorization. Suppose that $a,b,x\in S(\tau)$ and that $r>0$.
\begin{enumerate}
\item \quad Let $p,q>1$ satisfy $\frac{1}{p}+\frac{1}{q}=1$. If $\vert a\vert ^{rp}\in E(\tau)$ and $\vert b\vert ^{rq}\in E(\tau)$, then $\vert ab\vert ^r\in E(\tau)$ and
\begin{equation}
\label{eqnHa}
\Vert \ \vert ab\vert^r\Vert_{E(\tau)}
\leq  \Vert \ \vert a\vert ^{rp}\Vert_{E(\tau)}^{\frac{1}{p}}
\Vert \ \vert b\vert ^{rq}\Vert_{E(\tau)}^{\frac{1}{q}},
\end{equation}
\item \quad
If $\vert a^*ax\vert ^r \in E(\tau)$ and $\vert xbb^*\vert^r \in E(\tau)$ then 
\begin{equation}
\label{eqnHb}
\Vert \ \vert axb\vert^r\Vert_{E(\tau)}^2
\leq  \Vert \ \vert a^*ax\vert ^r\Vert_{E(\tau)}
\Vert \ \vert xbb^*\vert ^r\Vert_{E(\tau)}.
\end{equation}

\end{enumerate}

\end{cor}
\bigskip

The first inequality uses the identity $\mu(bb^*)=\mu(b^*b)$. The second
inequality in the case of trace ideals is due to
Bhatia and Davis ~\cite {BHDA1995} Theorem 1.

\begin{cor}
\label{corhiai2} Let  $E$ be a symmetric space on $[0,\infty)$ whose norm is monotone with respect to logarithmic submajorization and suppose that $0\leq a,b\in S(\tau)$, that $x\in S(\tau)$ and $r>0$. If $0\leq
\nu\leq 1$, and if $\vert ax\vert ^r, \ \vert xb\vert ^r\in E(\tau)$, then
$\vert a^\nu xb^{1-\nu}\vert ^r\in E(\tau)$
and 
\begin{equation}
\label{eqnHiai2}
\Vert  \vert a^\nu xb^{1-\nu}\vert ^r\Vert _{E(\tau)}
\leq \Vert \ \vert ax\vert ^r\Vert _{E(\tau)}^\nu 
\Vert \vert\ xb\vert ^r\Vert _{E(\tau)}^{1-\nu}
\end{equation}
\begin{equation}
\label{eqnHiai3}
\Vert  \vert a^\nu xb^{1-\nu}\vert ^r\Vert _{E(\tau)}
\Vert  \vert a^{1-\nu} xb^{\nu}\vert ^r\Vert _{E(\tau)}
\leq
\Vert \ \vert ax\vert ^r\Vert _{E(\tau)} 
\Vert \vert\  xb\vert ^r\Vert _{E(\tau)}
\end{equation}

\end{cor}

The inequality ~(\ref{eqnHiai2}) follows directly from
Proposition~\ref{propKIp4} by replacing $a,b$ by $a^\nu,b^{1-\nu}$ respectively,
and setting $p,p^\prime  =1/\nu$. The inequality ~\ref{eqnHiai3} then follows
immediately.

For $n\times  n$ matrices, the proposition which now follows is due to Hiai and Zhan
~\cite{HZ}. See also ~\cite{HSh} in the case of non-commutative $L_p$-spaces.

\begin{prop}
\label{prophiai3}
Let  $E$ be a  symmetric space on $[0,\infty)$ whose norm is monotone with respect to logarithmic submajorization and suppose
that $x\in S(\tau)$, that $0\leq a,b\in S(\tau)$ and that $r>0$. If $0\leq
\nu\leq 1$, and if $\vert ax\vert^r,\ \vert xb\vert^r \in E(\tau)$,
then the
function
\begin{equation*}
f(\nu)=
\Vert  \vert a^\nu xb^{1-\nu}\vert ^r\Vert _{E(\tau)}
\Vert  \vert a^{1-\nu} xb^{\nu}\vert ^r\Vert _{E(\tau)},\quad \nu\in[0,1],
\end{equation*}
is convex on the interval $[0,1]$ and attains its minimum at $\nu=1/2$ and its
maximum at the points $\nu=0,1$. In particular, for all $\nu\in [0,1]$,
\begin{equation}
\label{eqnlcon}
\Vert \  \vert a^{\frac{1}{2}}xb^{\frac{1}{2}}\vert^r\Vert _{E(\tau)}^2
\leq 
\Vert  \vert a^\nu xb^{1-\nu}\vert ^r\Vert _{E(\tau)}
\Vert  \vert a^{1-\nu} xb^{\nu}\vert ^r\Vert _{E(\tau)}
\leq \Vert \ \vert ax\vert ^r\Vert _{E(\tau)}\Vert \ \vert xb\vert ^r\Vert
_{E(\tau)}.
\end{equation}
In addition, if $a^{\frac{1}{2}}xb^{\frac{1}{2}}\neq 0$, then $f$ is logarithmically convex on $[0,1]$.
\end{prop}
\begin{proof} We observe first that it follows from (\ref{eqnHiai3}) that $f$ is bounded on $[0,1]$. Consequently, to show that $f$ is convex on $[0,1]$, it suffices to show that $f$ is mid-point convex. See, for example, \cite{Za}, Chapter 5, 25.3. To this end, suppose that $t\pm s\in 
[0,1]$ and 
using ~(\ref{eqnHb}), observe that
\begin{align*}
\Vert \ \vert a^txb^{1-t}\vert ^r\Vert _{E(\tau)}^2
&=\Vert \ \vert a^s(a^{t-s}xb^{1-t-s})b^s\vert ^r\Vert _{E(\tau)}^2\\
&\leq 
\Vert \ \vert a^{t+s}xb^{1-(t+s)}\vert ^r\Vert _{E(\tau)}\cdot
\Vert \  \vert a^{t-s}xb^{1-(t-s)}\vert ^r\Vert _{E(\tau)}
\end{align*}
and
\begin{align*}
\Vert \ \vert a^{1-t}xb^t\vert ^r\Vert _{E(\tau)}^2
&=\Vert \ \vert a^s(a^{1-t-s}xb^{t-s})b^s\vert ^r\Vert_{E(\tau)}^2\\
&\leq
\Vert \ \vert a^{1-(t-s)}xb^{t-s}\vert ^r\Vert _{E(\tau)}\cdot
\Vert \ \vert a ^{1-(t+s)}xb^{t+s}\vert ^r\Vert _{E(\tau)}.
\end{align*}
Consequently,
\begin{equation}
\label{eqncon}
2f(t)\leq 2\sqrt{f(t+s)f(t-s)}\leq f(t+s)+f(t-s).
\end{equation}

It now follows that $f((\cdot);a,b,x)$ is convex on $[0,1]$ whenever 
$\vert ax\vert^r,\ \vert xb\vert ^r
\in E(\tau)$. Since $f$ is clearly symmetric
about the point $1/2$, it follows that  $f$ attains its minimal value over [0,1] at
$1/2$ and its maximal value at the points $0,1$, and this establishes the first part of the proposition.
Finally, if   $a^{\frac{1}{2}}xb^{\frac{1}{2}}\neq 0$, then it follows from the equation (\ref{eqnlcon})
that the function $\log f$ is bounded on $[0,1]$.  It suffices to show that $\log f$ is midpoint convex. This, however, is immediate from the first inequality in (\ref{eqncon}).
\end{proof}

\begin{rem}
\label{rmkFS}We remark finally that there is a very substantial class of symmetric spaces $E\subseteq S(\tau)$ whose norm is monotone with respect to submajorization. Indeed, if $E$ is {\it fully symmetric} in the sense that, if $x\in E, y\in S(\tau)$ satisfy $y\prec\prec x$ then $y\in E$ and $\Vert y\Vert _E \leq \Vert x\Vert _E$, then it is easily seen via remark~\ref{rmkLS} that
the norm on $E$ is monotone with respect to logarithmic submajorization. Such spaces are precisely the exact interpolation spaces for the pair $(L_1(\tau),{\mathcal M})$. See~\cite{DDP1992},~\cite{DDP1993}. If $E\subseteq S(m)$ is a fully symmetric space on $[0,\infty)$, then it is follows directly that $E(\tau)\subset S(\tau)$ is fully symmetric. Important special cases occur if $E\subseteq S(m)$ is one of the familiar $L_p$-spaces, Orlicz spaces or  Lorentz spaces. In the case that ${\mathcal M}=B({\mathcal H})$, then all trace ideals in the sense of Gohberg and Krein~\cite {GK} are fully symmetric. While it is well-known~\cite{SSS} that there are symmetric spaces on the $[0,\infty)$  which are not closed subspaces of any fully symmetric space,  we do not have an explicit example of a symmetric space on $[0,\infty)$ whose norm is monotone with respect to logarithmic submajorization
but which is not fully symmetric. This seems to be a difficult problem.

\end{rem}

 \section{A Weyl-type inequality for uniform majorisation and H\"older inequalities}

The following direct consequence of the Weyl inequality given in Theorem~\ref{thmFK} and the implication (i)$\Longrightarrow$ (iii) of Proposition~\ref{propBr} extends  
\cite[Theorem 4.2(iii)]{FaKo1986}. 

\begin{thm}\label{WIp1} If $x,y\in L_{\log_+}(\tau)$, and if $f:(0,\infty)\to(0,\infty)$ is any  continuous
increasing  function $f:[0,\infty)\to [0,\infty)$ for which $f(0)=0$ and
$f\circ\exp$ is convex, then
$$f(|xy|)\prec\prec f(\mu(x)\mu(y)).$$
\end{thm}

We strengthen this as follows, using the notion of {\it uniform Hardy-Littlewood
majorization} introduced in ~\cite{KS};see also \cite{LSZ} Chapter 3.4. If $x,y\in S(\tau)$ then $x$ is said to be {\it
strongly majorized by } $y$, written $x\lhd y$, if there exists $\lambda >0 $ such that
\begin{equation*}
\int_{\lambda a}^b\mu(x)dm \leq \int _a^b\mu(y)dm,\quad 0<\lambda a\leq b.
\end{equation*}
It is the case that $x\prec\prec y $ if $x\lhd y$, but the converse is not valid. We shall
need the following fundamental result, which is obtained by combining ~\cite{KS}
Proposition 4.3 with \cite{KS}Theorem 6.3.

\begin{thm}
\label{thmKS} Let $E\subseteq S(m)$ be a symmetric Banach space on $[0,\infty)$. If $y\in
E(\tau)$ and if $x\in S(\tau)$ satisfies $x\lhd y$ then $x\in E(\tau)$ and 
$\Vert x\Vert _{E(\tau)}\leq \Vert y\Vert _{E(\tau)}$
\end{thm}

Our proof of the following theorem uses the approach from the proof of 
\cite{S} Theorem 7.

\begin{thm} 
\label{thmWI} If $x,y\in L_{\log_+}(\tau)$ and  $f:[0,\infty)\to[0,\infty)$ is an increasing and continuous function for which $f(0)=0$ and 
 $f\circ\exp$ is convex, then
$$f(|xy|)\lhd f(\mu(x)\mu(y)).$$
\end{thm}
\begin{proof} 
Let $a,b$ be positive scalars with $0\leq 2a<b$. It will suffice 
to show that
$$
\int_{2a}^b f(\mu(xy)dm\leq \int_a^bf(\mu(x)\mu(y))dm.
$$
Without loss of generality, it may be assumed that ${\mathcal M}$
is non-atomic.  By ~\cite{DDP1989a} Lemma 2.4), there exist projections $p,q\in P({\mathcal M})$
with $\tau(p)=\tau(q)=a$ such that
$$
e^{\vert x\vert}\left ( (\mu(a;x),\infty)\right)\leq p
\leq 
e^{\vert x\vert}\left ( [\mu(a;x),\infty)\right),$$
$$e^{\vert y\vert}\left( (\mu(a;y),\infty)\right)\leq q
\leq 
e^{\vert y\vert}\left( [\mu(a;y),\infty)\right),
$$ 
$$
\mu(\vert x\vert p)=\mu(x)\chi_{[0,\tau(p))}=\mu(x)\chi_{[0,a)},\quad 
\mu(\vert y\vert q)=\mu(x)\chi_{[0,\tau(q))}=\mu(y)\chi_{[0,a)},
$$
and
$$
\mu(\vert x\vert ({\bf 1}-p))=\mu(\cdot +a,x),\quad
\mu(\vert y\vert ({\bf 1}- q))=\mu(\cdot +a, y)
$$
Observe that, for all $t>2a$,
$$
\mu(t;xy)\leq \mu(a;xpy)+\mu(a;x({\bf 1}-p)yq)
+\mu(t-2a;x({\bf 1}-p)y({\bf 1}-q)).
$$
Using~\cite{FaKo1986} Lemma 2.5 (vii), observe that, for all $\epsilon>0$,
\begin{equation*}
\mu(a+\epsilon;xpy)
\leq \mu(a+\epsilon;\vert x\vert py)\leq \mu(a;\vert x\vert p)\mu(\epsilon;y)=0
\end{equation*}
and
\begin{equation*}
\mu(a+\epsilon;x({\bf 1}-p)yq)
\leq \mu(\epsilon;x({\bf 1}-p))\mu(a;\vert y\vert q)=0.
\end{equation*}
By right-continuity of the singular value function, it follows that
$$
\mu(a;xpy)=\mu(a;x({\bf 1}-p)yq)=0
$$
and so
$$
\mu(t;xy)\leq 
\mu(t-2a;x({\bf 1}-p)y({\bf 1}-q))\quad t>2a.
$$
Consequently,
\begin{align*}
\int_{2a}^b f(\mu(xy))dm
&\leq \int_{2a}^b f(\mu(\cdot-2a;x({\bf 1}-p)y({\bf 1}-q)dm
\leq \int_0^{b-2a}f(\mu(x({\bf 1}-p)y({\bf 1}-q))dm\\
&\leq \int_0^{b-2a}f(\mu(x({\bf 1}-p))\mu(y({\bf 1}-q))dm \quad 
\quad {\rm using\ Theorem~\ref{WIp1}} \\
&\leq 
%\int_0^{b-2a}f(\mu(x({\bf 1}-p))\mu(y({\bf 1}-q)))dm 
\int_0^{b-2a} f(\mu(\cdot  +a;x)\mu(\cdot+a ;y))dm\\
&=\int  _a^{b-a}f(\mu(x)\mu(y))dm 
\leq \int  _{a}^{b}f(\mu(x)\mu(y))dm 
\end{align*}

This concludes the proof.
\end{proof}

\begin{cor}
\label{corb1} (cf.~\cite{BO}, Lemma 4) If $x,y\in L_{\log_+}(\tau)$
and  $r>0$ then
$$
\vert xy\vert ^r\lhd \mu(x)^r\mu(y)^r.
$$
\end{cor}

\begin{proof}
The inequality follows immediately from Theorem~\ref{thmWI} by taking $f(t)=t^r, \
t>0$. 

\end{proof}

A general  H\"older inequality for arbitrary symmetric spaces now follows.
\begin{cor}
\label{corHo}
Suppose that $E\subseteq S(m)$ is a  symmetric space, that 
$r,p,q>0$ and that $1/r=1/p+1/q$. If $\vert x\vert ^p,\ \vert
y\vert ^q\in E(\tau)
$, then
$\vert xy\vert ^r\in E(\tau)$ and 
$$
\Vert \vert xy \vert ^r\Vert _{E(\tau)}^{\frac{1}{r}}
\leq \Vert \vert x \vert ^p\Vert _{E(\tau)}^{\frac{1}{p}}
\Vert\  \vert y \vert ^q\Vert _{E(\tau)}^{\frac{1}{q}}
$$
\end{cor}
\begin{proof} 
By assumption, $\mu(x)^p\in E,\mu(y)^q\in E $ and so by Proposition~\ref{propHo}
$\mu(x)^r\mu(y)^r\in E$. The assumptions $\vert x\vert ^p,\ \vert y\vert ^q \in E(\tau)\subseteq L^1(\tau)+{\mathcal M}$ imply further that $\vert x\vert ^p,\vert y\vert ^p\in L_{\log_+}(\tau)$ and so also $x,y\in L_{\log_+}(\tau)$.

By Corollary \ref{corb1}, the uniform submajorization
$$
\vert xy\vert ^r\lhd \mu(x)^r\mu(y)^r
$$
together with Theorem~\ref{thmKS} yields that $\vert xy\vert ^r\in E$ and 
$$
\Vert \vert xy\vert ^r\Vert _{E(\tau)}\leq \Vert \mu(x)^r\mu(y)^r\Vert _E
$$
The assertion of the corollary now follows from Proposition~\ref{propHo}.

\end{proof}

The preceding Corollary~\ref{corHo} is given in ~\cite{BO} Corollary 1 and in \cite{S} Theorem 1 in the case that $r=1$. The
present approach  
is much more direct than that given in ~\cite{BO}.

Let $E\subseteq S(m)$ be a symmetric space and let $x\in S(\tau)$.  For ease of presentation, it will 
now be convenient to write $\Vert x\Vert _{E(\tau)}<\infty$ if $x\in E(\tau)$ and 
 $\Vert x\Vert _{E(\tau)}=+\infty$ otherwise. 
\begin{prop}
\label{propHo2} 
Suppose that $E\subseteq S(m)$ is a symmetric space.
Let $r>0$ and suppose that $1<p,q<\infty$ satisfy $1/p+1/q=1$.
Let $\phi,\psi:(0,\infty)\to (0,\infty) $ be  Borel functions 
such that $\phi(\lambda)\psi(\lambda)=\lambda ,\ \lambda>0$. If $a,b,x\in
S(\tau)$, then
\begin{equation}
\label{eqnHo2}
\Vert \vert a^*xb\vert ^r\Vert _{E(\tau)}
\leq \Vert \left (a^*\phi^2(\vert x^*\vert)a\right)^{\frac{pr}{2}}\Vert
_{E(\tau)}^{\frac{1}{p}}
 \Vert \left (b^*\psi^2(\vert x\vert)b\right)^{\frac{qr}{2}}\Vert
_{E(\tau)}^{\frac{1}{q}},
\end{equation}
provided the right hand side is finite.

\end{prop}
\begin{proof} Let $a,x,b\in S(\tau)$ and note that it follows from
Lemma~\ref{lemGMMN} that 
\begin{equation*}
a^*xb=a^*\phi(\vert x^*\vert ) u \psi(\vert x\vert) b,
\end{equation*}
where $x=u\vert x\vert$ is the polar decomposition.  Theorem~\ref{thmWI} with
$f(\cdot)=(\cdot)^r$ implies that
\begin{align*}
\vert a^*xb\vert ^r
&\lhd \mu^r(a^*\phi(\vert x^*\vert) u)\mu^r(\psi(\vert x\vert )b)
\\
&\leq \mu^r(\phi(\vert x^*\vert )a)\mu^r(\psi(\vert x\vert) b)
=\mu( (a^*\phi^2(\vert x^*\vert )a)^{\frac{r}{2}})
\mu((b^*\psi^2(\vert x\vert )b)^{\frac{r}{2}}).
\end{align*}
It now follows that
\begin{align*}
\Vert \vert a^*xb\vert ^r\Vert 
&\leq \Vert \mu( (a^*\phi^2(\vert x^*\vert )a)^{\frac{r}{2}})
\mu((b^*\psi^2(\vert x\vert )b)^{\frac{r}{2}})\Vert _E\\
&\leq 
\Vert  \mu^p( (a^*\phi^2(\vert x^*\vert )a)^{\frac{r}{2}})\Vert _E^{\frac{1}{p}}
\Vert \mu^q((b^*\psi^2(\vert x\vert )b)^{\frac{r}{2}})\Vert _E^{\frac{1}{q}}\\
&=
\Vert \left (a^*\phi^2(\vert x^*\vert)a\right)^{\frac{pr}{2}}\Vert
_{E(\tau)}^{\frac{1}{p}}
 \Vert \left (b^*\psi^2(\vert x\vert)b\right)^{\frac{qr}{2}}\Vert
_{E(\tau)}^{\frac{1}{q}},
\end{align*}
and this completes the proof of the proposition.
\end{proof}

For the special case that 
$\phi(\lambda)=\lambda ^{1/2}=\psi(\lambda),\ \lambda >0$ , we obtain the
following inequality. See ~\cite{BO} Corollary 2.

\begin{cor}
\label{corHo3} Suppose that $E\subseteq S(m)$ is a symmetric space. 
Let $r>0$ and suppose that $1<p,q<\infty$ satisfy $1/p+1/q=1$. If $a,b,x\in
S(\tau)$,then
\begin{equation}
\label{eqnHo3}
\Vert \vert a^*xb\vert ^r\Vert _{E(\tau)}
\leq \Vert \left (a^*\vert x^*\vert a\right)^{\frac{pr}{2}}\Vert
_{E(\tau)}^{\frac{1}{p}}
 \Vert \left (b^*\vert x\vert b\right)^{\frac{qr}{2}}\Vert
_{E(\tau)}^{\frac{1}{q}},
\end{equation}
whenever the right hand side is finite.
\end{cor}

A standard argument (see, for example ~\cite {BO} Theorem 4, or ~\cite{Al} 
Theorem 21) now yields the following consequence.

\begin{cor}
\label{corHo2}
Suppose that $E\subseteq S(m)$ is symmetrically normed.
Let $r>0$ and suppose that $1<p,q<\infty$ satisfy $1/p+1/q=1$.
Let $\phi,\psi:(0,\infty)\to (0,\infty) $ be  Borel functions 
such that $\phi(\lambda)\psi(\lambda)=\lambda ,\ \lambda>0$. If $a_i,x_i,b_i\in
S(\tau),\ 1\leq i\leq n$, then
\begin{equation}
\label{eqnH022}
\Vert \vert \sum_{i=1}^na_i^*x_ib_i\vert ^ r\Vert _{E(\tau)}
\leq \left\Vert \left(
\sum_{i=1}^na_i^*\phi^2(\vert x_i^*\vert )a_i\right)^{\frac{pr}{2}}
\right\Vert _{E(\tau)}^{\frac{1}{p}}
\ 
\left\Vert \left(\sum_{i=1}^nb_i^*\psi^2(\vert x_i\vert)b_i\right)^{\frac{qr}{2}}
\right\Vert ^{\frac{1}{q}},
\end{equation}
whenever the right  hand side is finite.

\end{cor}

We remark that Proposition~\ref{propHo2} and Corollary~\ref{corHo2} may be found
in ~\cite{Al}, Theorems 19 and  21 respectively, for the special case that
${\mathcal M}=B(H)$. See also ~\cite{Ki1999}.  In the present setting of $\tau$-measurable operators, the
special case of Proposition~\ref{propHo2} and Corollary~\ref{corHo2} obtained by
setting $\phi(\lambda)=\lambda ^{1/2}=\psi(\lambda),\ \lambda >0$ is proved in
~\cite{BO}.  However, the present approach via Theorem ~\ref{thmWI} is more
direct and transparent.

\def\cprime{$'$} \def\cprime{$'$} \def\cprime{$'$}
\providecommand{\bysame}{\leavevmode\hbox to3em{\hrulefill}\thinspace}


\begin{thebibliography}{10}

\bibitem{Al} 
H. Albadawi, \emph{H\"older type inequalities involving unitarily invariant 
norms}, Positivity \textbf {16} (2012) 255-270.

\bibitem{BO} 
T.Bekjan and M.Ospanov, \emph{H\"older type inequalities of measurable
operators}, Positivity, (2016) (to appear). 

\bibitem{Bh} R. Bhatia, \emph{Matrix Analysis}, Graduate Texts in Mathematics,
{\bf 169} Springer-Verlag New York, 1996.

\bibitem{BHDA1995}
R.Bhatia and C.Davis, \emph {A Cauchy-Schwartz inequality for operators with
applications}, Lin. Alg. Appl.\textbf{223} (1995)119-129.

\bibitem{Br}
L.G.Brown,\emph{Lidskii's theorem in the type II case}, in {\it Proc.U.S.-Japan Seminar, Kyoto 1983}, Pitman Research Notes, Math.Ser.,{\bf 123} (1986),3-19.
 
\bibitem{Di}
J.~Dixmier, \textit{ von Neumann Algebras}, Mathematical Library
{\bf 27}. Amsterdam: North Holland (1981).
 
\bibitem{DDP1993}
P.~G. Dodds, T.~K.-Y. Dodds, and B.~de~Pagter, \emph{Noncommutative {K}\"othe
duality}, Trans.Amer.Math.Soc.\textbf{339}(1993)717--750.
 
\bibitem{DDP1992}
P.~G. Dodds, T.~K.-Y. Dodds, and B.~de~Pagter, \emph{Fully symmetric operator
spaces}, Integr Equat. Oper. Theory\textbf{15}(1992)942--972.

\bibitem{DDP1989a}
P.~G. Dodds, T.~K.-Y. Dodds, and B.~de~Pagter, \emph{Noncommutative Banach function
spaces}, Math.Zeit.\textbf{201}(1989)583-597.

\bibitem{DDS}P.G.Dodds, T.K.Dodds and F.A. Sukochev, \emph {on $p$-convexity and
$q$-concavity in non-commutative symmetric spaces}, Integral Equations and
Operator Theory \textbf{78} (2014)91-114.

\bibitem{DP2} P.G. Dodds and B.de Pagter,
 \textit{Normed K\"othe spaces: A non-commutative viewpoint,} Indagationes
 Mathematicae {\bf 25} (2014) 206-249.
 

\bibitem{DPS} P.G.Dodds, B.de Pagter and F. Sukochev, \textit{Theory of 
Noncommutative Integration},
unpublished monograph, to appear.

\bibitem{DySk} K.Dykema and A.Skripka, \emph{H\"olders inequality for roots of
symmetric operator spaces}, Studia Math. \textbf {228}(2015) 47-54

\bibitem{Fa} T.Fack, \emph{ Sur la notion de valeur caract\'eristique},
  J. Operator Theory {\bf 7} (1982) 307-333.
  
  \bibitem{Fa83}
  T.~ Fack, \emph{ Proof of the conjecture of A.Grothendieck on the
  Fuglede-Kadison determinant},
  J.Functional Anal. \textbf{50} (1983) 215-228.
  \bibitem{FK} B.Fuglede and R.V.Kadison, \emph{Determinant theory in finite factors}, Ann. of Math.{\bf 55} (1952) 520-530.
   
   \bibitem{FaKo1986}
T.~Fack and H.~Kosaki, \emph{Generalized {$s$}-numbers of {$\tau$}-measurable
  operators}
, Pacific J. Math. \textbf{123} (1986), no.~2, 269--300.

\bibitem{GK}
I.C.Gohberg and M.G.Krein, \emph{Introduction to the theory of non-selfadjoint operators}, Transl.Math.Monographs, vol. 18, Amer.Math.Soc., 
1969.Providence, R.I.,


 \bibitem{GMMN}
 F.Gesztesy, M. Malamud, M. Mitrea and S. Naboko,
 \emph{Generalized Polar Decompositions for Closed Operators in Hilbert spaces
 and some Applications} Integral Equations Operator Theory \textbf{64}(2009)
 83-113.
 
 \bibitem{G}
 A.~Grothendieck, \emph{R\'earrangements de fonctions et in\'eqalit\'es de convexity dans les alg\`ebres de von Neumann munies d'une trace}
 Seminaire N.Bourbaki, 1954-1956 exp.no. 113, 127-139.
 
 
\bibitem{HS}
U. Haagerup and H. Schultz,
\textit{Brown measures of unbounded operators affiliated with a finite von Neumann algebra}, Math.Scand.{\bf 100} (2007), 209-263.
 
 \bibitem{Ha}
 Y. Han,  \emph{On the Araki-Lieb-Thirring  inequality 
 in the semifinite von
 Neumann algebra}, Ann.Funct.Anal  {\bf 7} (2016) 622-635.
 
 \bibitem{HSh} Y. Han and J. Shao, \emph{Some convexity inequalities in
 non-commutative $L_p$-spaces}, Journal of Inequalities and
 Applications,\textbf{385} (2014) 6pp.
 
 \bibitem{Hi1997} 
 F.Hiai, \emph{Log-majorization and norm inequalities for exponential
 operators}, in "Linear Operators" Banach Center Publ. vol. 38 pp.119-181,
 Polish Academy of Sciences, Warszawa, 1997.

 \bibitem{HZ}  
 F.Hiai and X.Zhan, \emph{Inequalities involving unitarily invariant norms and
 operator monotone functions}, Linear algebra and its Applications
 \textbf{341} (2002) 151-169.
 
\bibitem{Ki1999}
 F.Kittaneh, \emph{Some norm inequalities for operators}, Canad. Math. Bull.
 \textbf {42} (1999)87-96.
 
 \bibitem{KPS}
 S.G. Krein, Y.I. Petunin and E.M. Sememov,\emph{Interpolation of linear operators}, In: Translations of Mathematical Monographs, vol. 54, American Mathematical Society, Providence (1982).
 
 \bibitem{KS} N.J. Kalton and F.A. Sukochev,\emph{Symmetric norms and spaces of operators}, J.Reine
 Angew.Math \textbf {621}(2008) 81-121.

 
 \bibitem{Ko1998}
 H. Kosaki,  \emph{Arithmetic-geometic mean and related inequalities for
 operators},
 J.Functional Anal.\textbf{156} (1998) 426-451.

 
\bibitem{Ko1992} 
  H. Kosaki, \emph{An Inequality of Araki-Lieb-Thirring},
  Proc.Amer.Math.Soc.\textbf{114} (1992),477-481.
 
  \bibitem{LSZ}
 S. Lord, F.Sukochev and D. Zanin, \emph{Singular Traces: Theory and Applications} De Gruyter
Studies in Mathematics, 46, De Gruyter, Berlin, (2013).
  
\bibitem{Ne}E. Nelson, \textit{Notes on non-commutative integration}, J.
Functional Anal. {\bf 15}(1974), 103-116.

\bibitem{W} H.Weyl, \textit{Inequalities between the two kinds of eigenvalues of a linear transformation},Proc. nat. Acad. Sci {\bf 35} (1949),408-411.
  
  \bibitem{S} F.Sukochev,  {\it H\"older inequality for symmetric operator spaces 
  and trace property of K-cycles}, Bull. London Math. Soc. {\bf 48} (2016) 637-647.
 
 \bibitem{SSS} A.A.Sedaev, E.M. Semenov and F.A.Sukochev, 
 \textit{Fully symmetric function spaces without an equivalent Fatou norm}, Positivity {\bf 19} (2015) 419-437.
 
\bibitem{Ta} M. Takesaki, \textit {Theory of Operator Algebras I},
Springer-Verlag, New-York-Heidelberg-Berlin, 1979.


\bibitem{Te} M. Terp, \textit{$L^p$-spaces associated with von Neumann
algebras}, Notes, Copenhagen University (1981).


\bibitem{WG}
B.-Y. Wang and M.-P. Gong, 
\textit{Some eigenvalue inequalities for positive semidefinite matrix power
products}, Linear Algebra Appl. {\bf 184} (1993)249-260

\bibitem{Za}
A.C.Zaanen, \textit{Integration}, North-Holland, Amsterdam, 1967. 


\end{thebibliography}
\end{document}